\documentclass[
final
]{dmtcs-episciences}

\usepackage[utf8]{inputenc}
\usepackage{subfigure}
\newtheorem{theorem}{Theorem}[section]

\newtheorem{lemma}[theorem]{Lemma}

\usepackage[round]{natbib}

\author[Deyu Wu et. al]{Deyu Wu\affiliationmark{1}
  \and Yipei Zhang\affiliationmark{2}\thanks{Corresponding author. Email address: zyipei@163.com}
  \and Xiumei Wang\affiliationmark{3}}
\title[A note on removable edges in near-bricks]{A note on removable edges in near-bricks}
\affiliation{
  Department of Basic Courses, Zhengzhou University of Science and Technology, Zhengzhou, China\\
  School of Mathematics and Statistics, North China University of Water Resources and Electric Power, Zhengzhou, China\\
  School of Mathematics and Statistics, Zhengzhou University, Zhengzhou, China}
\begin{document}
\publicationdata{vol. 26:2}{2024}{8}{10.46298/dmtcs.11747}{2023-08-21; 2023-08-21; 2024-03-05}{2024-04-10}
\maketitle
\begin{abstract}
  An edge $e$ of a matching covered graph $G$ is removable if $G-e$ is also matching covered. Carvalho, Lucchesi, and Murty showed that every brick $G$ different from $K_4$ and $\overline{C_6}$ has at least $\Delta-2$ removable edges, where $\Delta$ is the maximum degree of $G$. In this paper, we generalize the result to irreducible near-bricks, where a graph is irreducible if it contains no single ear of length three or more.
\end{abstract}

%

\section{Introduction}
\label{sec:in}
All the graphs considered in this paper may have multiple edges, but no
loops. We follow \cite{BM08} for undefined notations and terminologies.
Let $G$ be a graph with the vertex set $V(G)$ and the edge set $E(G)$.
We denote by $\Delta(G)$, or simply $\Delta$, the maximum degree of the vertices of $G$.
For a subset $X$ of $V(G)$, let $G[X]$ denote the subgraph of $G$ induced by $X$.
A \emph{matching} of a graph is a set of pairwise nonadjacent edges.
A \emph{perfect matching} is one which covers every vertex of the graph.
A nontrivial connected graph is \emph{matching covered} if each edge lies in a perfect matching of the graph. Clearly, every matching covered graph different from $K_2$ is 2-connected.

For a nonempty proper subset $X$ of $V(G)$, let $\partial(X)$ denote the set of all the edges of $G$ with one end in $X$ and the other end in $\overline{X}$, where $\overline{X}:=V(G)\setminus X$.
The set $\partial(X)$ is called a {\it cut} of $G$, the sets $X$ and $\overline{X}$ its shores.
The shore $X$  of $\partial(X)$  is {\it bipartite} if the induced subgraph $G[X]$ is bipartite.
A cut is {\it trivial} if one of its shores is a singleton, and is \emph{nontrivial} otherwise.
We denote by $G/(X\rightarrow x)$, or simply $G/X$, the graph obtained from $G$ by shrinking $X$ to a single vertex $x$.
Similarly, we denote by $G/(\overline{X}\rightarrow \overline{x})$, or simply $G/\overline{X}$,  the graph obtained from $G$ by shrinking $\overline{X}$ to a single vertex $\overline{x}$.
The two graphs $G/X$ and $G/\overline{X}$ are called the two $\partial(X)$-{\it contractions} of $G$.

Let $G$ be a matching covered graph. A cut $C$ of $G$ is \emph{tight} if $|M\cap C|=1$ for each perfect matching $M$ of $G$. A matching covered graph which is free of nontrivial tight cuts is a \emph{brace} if it is bipartite, and is a \emph{brick} otherwise.
If $G$ has a nontrivial tight cut $C$, then each $C$-contraction of $G$ is a matching covered graph that has strictly fewer vertices than $G$.
Continuing  in this way, we can obtain a list of matching covered graphs without nontrivial tight cuts, which are bricks and braces.
This procedure is known as a {\it tight cut decomposition} of $G$.
In general, a matching covered graph may admit several tight cut decompositions.
\cite{Lovasz1987} showed  that any two tight cut decompositions of a matching covered graph yield the same list of bricks and braces (up to multiple edges).
This implies that the number of bricks is uniquely determined by $G$.
Let $b(G)$ denote the number of the bricks of $G$.
Note that $b(G)=0$ if and only if $G$ is bipartite.

A graph $G$ is a {\it near-brick} if it is a matching covered graph with $b(G)=1$. Clearly, a near-brick is 2-connected and a brick is a near-brick. 
A {\it single ear} of a graph is a path of odd length whose internal vertices (if any) all
have degree two in this graph.
A graph is {\it irreducible} if it contains no single ear of length three or more.
\cite{ELP1982} proved that a graph $G$ is a brick if and only if $G$ is 3-connected and $G-x-y$ has a perfect matching for any two distinct vertices
$x,y\in V(G)$.
Therefore, a brick is irreducible. However, a near-brick is not necessarily irreducible.
For instance, subdividing an edge of a graph in Figure \ref{fig1} by  inserting two vertices results in a near-brick, which is not irreducible.

\begin{figure}[h]
 \centering
 \includegraphics[width=0.6\textwidth]{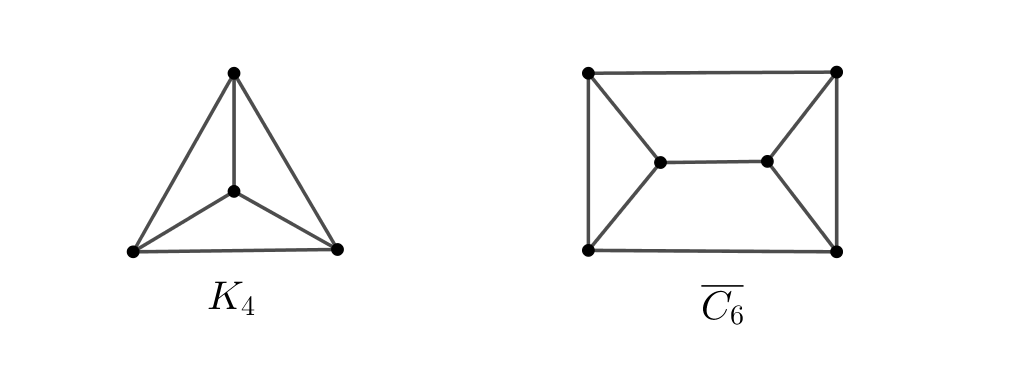}\\
 \caption{The two bricks.}\label{fig1}
\end{figure}

An edge $e$ of a matching covered graph $G$ is {\it removable} if $G-e$ is also matching covered, and is {\it nonremovable} otherwise.
Clearly, each multiple edge of a matching covered graph is in fact a removable edge.
The notion of removable edge is related to ear decompositions of
matching covered graphs introduced by Lov\'{a}sz and Plummer. \cite{Lovasz1987} showed that every brick distinct from $K_4$ and $\overline{C_6}$ has a removable edge, where $K_4$ and $\overline{C_6}$ are shown in Figure \ref{fig1}. Carvalho, Lucchesi, and Murty proved the following stronger result.

\begin{theorem}[\cite{CLM1999}]\label{brick}
Every brick $G$ different from $K_4$ and $\overline{C_6}$ has at least $\Delta-2$ removable edges.
\end{theorem}

The following theorem is our main result which generalizes the above theorem to irreducible near-bricks.

\begin{theorem}\label{main:result}
Every irreducible near-brick $G$ different from $K_4$ and $\overline{C_6}$ has at least $\Delta-2$ removable edges.
\end{theorem}

The paper is organized as follows. In Section \ref{sec:Preliminaries}, we present some basic results. In Section \ref{sec:Proof}, we give a proof of Theorem \ref{main:result}.

\section{Preliminaries}
\label{sec:Preliminaries}
\begin{lemma}[\cite{CLM1999}]\label{brace}
In a brace on six or more vertices, every edge is removable.
\end{lemma}

\begin{lemma}[\cite{FKC2021}]\label{brace-3-connected}
Every brace on six or more vertices is 3-connected.
\end{lemma}

\begin{lemma}[\cite{ZWY2022}]\label{removable-tight-cut}
Let $C$ be a tight cut of a matching covered graph $G$ and $e$ an edge of $G$. Then $e$ is
removable in $G$ if and only if $e$ is removable in each $C$-contraction of G which contains it.
\end{lemma}

The following equality reveals an important property of the numbers of bricks of matching covered graphs with respect to tight cuts.

\begin{lemma}[\cite{CLM2002}]\label{tight-cut-b}
Let $G$ be a matching covered graph and $C$ a tight cut of $G$. Let $G_1$
and $G_2$ be the two $C$-contractions of $G$. Then $b(G)=b(G_1)+b(G_2)$.
\end{lemma}

Using the above lemma, we can easily obtain the following result, also see \cite{CLM2002}.

\begin{lemma}[\cite{CLM2002}]\label{tight-cut-near-brick}
For any tight cut $C$ of a near-brick $G$, precisely one of the shores of $C$ is bipartite.
\end{lemma}

{\it To bisubdivide an edge $e$} of a graph $G$ is to replace $e$ by an odd path with length at least three. 
The resulting graph is called a {\it  bisubdivision} of $G$ at the edge $e$.
Let $RE(G)$ denote the set of all the removable edges of $G$.

\begin{lemma}\label{bisubdivision}
Let $G$ be a graph and let $H$ be a bisubdivision of $G$ at an edge $e$. Suppose that $H$ is a matching covered graph. Then $G$ is a matching covered graph with $b(G)=b(H)$ and $RE(H)= RE(G)\backslash\{e\}$.
\end{lemma}

\begin{proof}
Since $H$ is a bisubdivision  of $G$ at the edge $e$, $H$ is obtained from $G$ by replacing $e$ by an odd path $P$ with length at least three.
We assert that $G$ is not isomorphic to $K_2$. Otherwise, $H$ is an odd path, contradicting the assumption that $H$ is a matching covered graph.
Let $e=uv$ and $X=V(P)\backslash\{v\}$. Then $G$ is isomorphic to $H/X$. Since $P-v$ is an even path of $H$ with all internal vertices of degree 2 in $H$, for each perfect matching $M$ of $H$, we have $|M\cap \partial_H(X)|=1$. Then $\partial_H(X)$ is a tight cut of $H$.
Since $H$ is a matching covered graph, so does $G$.
Since  $G$ is not isomorphic to $K_2$,  $G$ is 2-connected. So $u$ has at least two neighbours in $G$. This implies that the underlying simple graph of $H/\overline{X}$ is an even cycle. So $b(H/\overline{X})=0$.
By Lemma \ref{tight-cut-b}, we have $b(H)=b(G)+b(H/\overline{X})=b(G)$.

Now we proceed to show that $RE(H)=RE(G)\backslash\{e\}$.
Note that each edge of $P$ is incident with a vertex of degree 2 in $H$, and hence is nonremovable in $H$.
Thus, if $f\in RE(H)$, then $f\in E(G)\backslash\{e\}$. By Lemma \ref{removable-tight-cut}, we have $f\in RE(G)\backslash\{e\}$. So $RE(H)\subseteq RE(G)\backslash\{e\}$. Now assume that $f\in RE(G)\backslash\{e\}$. If $f\notin \partial_H(X)$, Lemma \ref{removable-tight-cut} implies that $f\in RE(H)$. If $f\in \partial_H(X)$, then $f$ is incident with $u$ in $H$. Since $f$ is removable in $G$, we have  $d_G(u)\geq3$. So $d_H(u)=d_G(u)\geq3$. It follows that $f$ is a multiple edge of $H/\overline{X}$, and then is a removable edge of $H/\overline{X}$. Again by Lemma \ref{removable-tight-cut}, we have $f\in RE(H)$. It follows that $RE(G)\backslash\{e\}\subseteq RE(H)$. Consequently, $RE(H)=RE(G)\backslash\{e\}$.
\end{proof}

\section{Proof of Theorem \ref{main:result}}
\label{sec:Proof}
Suppose that $G$ is an irreducible near-brick different from $K_4$ and $\overline{C_6}$, and $\Delta=\Delta(G)$.
Then $b(G)=1$ and $|V(G)|\geq4$. Moreover, $G$ is  2-connected and matching covered.
So $\delta(G)\geq2$. If $\Delta<3$, then each vertex of $G$ has degree two. Thus $G$ is an even cycle. This implies that $b(G)=0$, a contradiction. Therefore, $\Delta\geq3$. We shall show that $G$ has at least $\Delta-2$ removable edges by induction on $|V(G)|+|E(G)|$.
Now we consider the following two cases according to whether $G$ has parallel edges or not.

{\bf Case 1.} $G$ has parallel edges.

Suppose that $G$ has two parallel edges, say $e_1,e_2$, which have common ends. Then $G-e_1$ is a near-brick, but it has strictly fewer edges than $G$.
Moreover, we have $\Delta(G-e_1)\geq\Delta-1$.
Recall that each multiple edge of a matching covered graph is a removable edge.
Then both $e_1$ and $e_2$ are removable in $G$, that is, $\{e_1,e_2\}\subseteq RE(G)$.

\vspace{2mm}
{\it Claim 1.} $RE(G-e_1)\subseteq RE(G)$.

Suppose that $f\in RE(G-e_1)$. If $f=e_2$, then $f\in RE(G)$.
If $f\neq e_2$, then $G-e_1-f$ is matching covered, and both $e_1$ and $e_2$ are multiple edges of $G-f$.
Therefore, $G-f$ is matching covered and then $f\in RE(G)$.
Claim 1 holds.
\vspace{2mm}

If $G-e_1$ is one of $K_4$ and $\overline{C_6}$, then $\Delta=4$ and $G$ has exactly two removable edges $e_1$ and $e_2$. The result holds. We may thus assume that $G-e_1$ is neither $K_4$ nor $\overline{C_6}$.
If $G-e_1$ is irreducible, by the induction hypothesis, $G-e_1$ has at least $\Delta(G-e_1)-2$ removable edges, that is, $|RE(G-e_1)|\geq \Delta(G-e_1)-2$.
Recall that $e_1\in RE(G)$ and $\Delta(G-e_1)\geq\Delta-1$. By Claim 1,
$$|RE(G)|\geq|RE(G-e_1)|+1\geq \Delta(G-e_1)-2+1\geq\Delta-2.$$
Therefore, $G$ has at least $\Delta-2$ removable edges.

If $G-e_1$ is not irreducible, since $G$ is irreducible, $G-e_1$ has a single ear with length at least 3, which  contains $e_2$. Let $P_{e_2}$ be such a maximal single ear, and $s$ and $t$ the two ends of $P_{e_2}$.
Since $P_{e_2}$ contains $e_2$ in $G-e_1$, $e_2$ is incident with a vertex of degree 2 in $G-e_1$. This implies that $e_2\notin RE(G-e_1)$.
Let $G'$ be the graph obtained from $G-e_1-(V(P_{e_2})\backslash\{s,t\})$ by adding a new edge $e$ that connects $s$ and $t$.
For each vertex $x^{*}$ of $G'$, we can see that $d_{G'}(x^{*})=d_{G-e_1}(x^{*})$.
In particular, if $x^{*}\notin\{s,t\}$, then  $d_{G'}(x^{*})=d_{G-e_1}(x^{*})=d_G(x^{*})$.
Moreover, we have $\Delta(G')=\Delta(G-e_1)$.

\vspace{2mm}
{\it Claim 2.} $G'$ is a near-brick and $RE(G-e_1)=RE(G')\backslash\{e\}$.

Note that $G-e_1$ is a bisubdivision of $G'$ at the edge $e$. Since $G-e_1$ is a near-brick, by Lemma \ref{bisubdivision}, $G'$ is a near-brick and $RE(G-e_1)=RE(G')\backslash\{e\}$. Claim 2 holds.
\vspace{2mm}

{\it Claim 3.} $G'$ is irreducible.

Assume that both $s$ and $t$ have degree two in $G'$. Then both $s$ and $t$ have degree two in $G-e_1$.
Recall that $G-e_1$ is a near-brick. Then $G-e_1$ is 2-connected and is not an even cycle. So we have $st\notin E(G-e_1)$.
Let $s_1$ be the only vertex in $N_{G-e_1}(s)\backslash V(P_{e_2})$ and $t_1$ the only vertex in $N_{G-e_1}(t)\backslash V(P_{e_2})$.
If $s_1=t_1$, then $G-e_1-s_1$ has an even component $P_{e_2}$.
This implies that $s_1$ is a cut vertex of $G-e_1$, contradicting the fact that $G-e_1$ is 2-connected. So $s_1\neq t_1$.
Then $P_{e_2}+ss_1+tt_1$ is a single ear of $G-e_1$ that contains $e_2$ and longer than $P_{e_2}$, contradicting the maximal of $P_{e_2}$.
Therefore, at least one of $s$ and $t$ have degree three or more in $G'$.
Assume, without loss of generality, that $s$ has degree three or more in $G'$.
If $t$ has degree three or more in $G'$, since $G$ is irreducible, so does $G'$.
The claim holds.
Now assume that $t$ has degree two in $G'$.
If $N_{G'}(t)=\{s\}$, then $st\in E(G-e_1)$. This implies that $s$ is a cut vertex of $G-e_1$, a contradiction. So $t$ has exactly two distinct neighbours in $G'$.
Let $t'$ be the only vertex in $N_{G'}(t)\backslash\{s\}$.
Then $t'$ is the only vertex in $N_{G-e_1}(t)\backslash V(P_{e_2})$.
If $t'$ has degree two in $G-e_1$, since $G-e_1$ is 2-connected, $t'$ has exactly one neighbour other than $t$, say $t''$, and  $t''\neq s$. This implies that $P_{e_2}+tt't''$ is a single ear of $G-e_1$ that contains $e_2$ and longer than $P_{e_2}$, contradicting the maximal of $P_{e_2}$.
So $t'$ has degree three or more in $G-e_1$, and then it has degree three or more in $G'$.
Since $s$ has degree three or more in $G'$ and $G$ is irreducible, $G'$ is irreducible.
Claim 3 holds.
\vspace{2mm}

If $G'$ is one of $K_4$ and $\overline{C_6}$, then $\Delta(G-e_1)=\Delta(G')=3$. So we have $\Delta\leq\Delta(G-e_1)+1=4$. Recall that $\{e_1,e_2\}\subseteq RE(G)$.
The result holds.
Now suppose that $G'$ is different from $K_4$ and $\overline{C_6}$.
By the induction hypothesis, $G'$ has at least $\Delta(G')-2$ removable edges, that is, $|RE(G')|\geq \Delta(G')-2$.
By Claim 2, we have $RE(G-e_1)=RE(G')\backslash\{e\}$.
Then $|RE(G-e_1)|\geq|RE(G')|-1\geq\Delta(G')-3$. 
By Claim 1, we have $RE(G-e_1)\subseteq RE(G)$.
Recall that $e_2\notin RE(G-e_1)$ and $\{e_1,e_2\}\subseteq RE(G)$.
Then $$|RE(G)|\geq|RE(G-e_1)|+2\geq\Delta(G')-3+2=\Delta(G-e_1)-1\geq\Delta-2.$$
That is, $G$ has at least $\Delta-2$ removable edges.

{\bf Case 2.} $G$ is simple.

Note that $C_4$ and $K_4$ are the only two simple matching covered graphs with four vertices. Since $G$ is a near-brick different from $K_4$ and $|V(G)|\geq4$, we have $|V(G)|\geq6$.

If $G$ is a brick, by Theorem \ref{brick}, the result holds. So we may assume that $G$ is not a brick.
Since $b(G)=1$, $G$ is not a brace and  $G$ has a nontrivial tight cut.
Let $C:=\partial(X)$ be a nontrivial tight cut of $G$.
By Lemma \ref{tight-cut-near-brick}, we may assume that $G[X]$ is bipartite, subject to this,  $X$ is minimal.
Let $G_1=G/(X\rightarrow x)$ and $G_2=G/(\overline{X}\rightarrow \overline{x})$.
Then $G_1$ is matching covered and $G_2$ is a brace.
By Lemma \ref{tight-cut-b}, $G_1$ is a near-brick.
Assume that $(B,I)$ is the  bipartition of $G_2$ such that $\overline{x}\in I$.
Then $X=B\cup (I\backslash\{\overline{x}\})$.
Let $u$ be a vertex of $G$ such that $d_{G}(u)=\Delta$, and write $\Delta_1=\Delta(G_1)$.

First suppose that $|V(G_2)|\geq 6$. Then $|I|\geq3$. By Lemma \ref{brace}, each edge of $G_2$ is removable in $G_2$. By Lemma \ref{removable-tight-cut}, each removable edge of $G_1$ is also a removable edge of $G$, that is, $RE(G_1)\subseteq RE(G)$.
Since $G_2$ is a brace and $|V(G_2)|\geq 6$, by Lemma \ref{brace-3-connected}, we have $\delta(G_2)\geq3$.
Then each vertex in $X$ has degree three or more in $G$, and $d_{G_1}(x)=d_{G_2}(\overline{x})\geq3$.
Since $G$ is irreducible, so does $G_1$. If $G_1$ is one of $K_4$ and $\overline{C_6}$, then $|C|=3$ and each vertex of $\overline{X}$ has degree three in $G$.
We may thus assume that $u\in X$.
We assert that $|\partial_G(u)\cap C|\leq 2$. Otherwise, $u$ is a cut vertex of $G$, contradicting the fact that $G$ is 2-connected.
Since every edge of $G_2$ is removable in $G_2$, by Lemma \ref{removable-tight-cut}, each edge of $G[X]$ is removable in $G$. This implies that $G$ has at least $|\partial_G(u)\backslash C|\geq \Delta-2$ removable edges.
Now assume that $G_1$ is different from $K_4$ and $\overline{C_6}$.  By the induction hypothesis, $G_1$ has at least $\Delta_1-2$ removable edges, that is, $|RE(G_1)|\geq\Delta_1-2$.

If $u\in \overline{X}$, then $\Delta_1\geq\Delta$.
Since $RE(G_1)\subseteq RE(G)$, we have
$$|RE(G)|\geq|RE(G_1)|\geq\Delta_1-2\geq\Delta-2.$$
So $G$ has at least $\Delta-2$ removable edges.
Now assume that $u\in X$. Recall that each edge of $G[X]$ is removable in $G$.
If $u\in I\backslash\{\overline{x}\}$, then each edge incident with $u$ is removable in $G$. This implies that $G$ has at least $\Delta$ removable edges.
If $u\in B$, since $G$ is simple and $u$ has at most $|I|-1$ neighbours in $I\backslash\{\overline{x}\}$, we have $|C|\geq d_G(u)-(|I|-1)=\Delta-|I|+1$.
Then $\Delta_1\geq d_{G_1}(x)=|C|\geq\Delta-|I|+1$.
Note that each edge which is incident with a vertex in $I\backslash\{\overline{x}\}$ is removable in $G$.
Since $RE(G_1)\subseteq RE(G)$ and $\delta(G_2)\geq3$, we have
$$|RE(G)|\geq|RE(G_1)|+3|I\backslash\{\overline{x}\}|\geq\Delta_1-2+3(|I|-1)\geq\Delta-|I|+1-5+3|I|=\Delta-4+2|I|\geq\Delta+2.$$
Therefore, $G$ has at least $\Delta+2$ removable edges.

Now suppose that $|V(G_2)|=4$.
Then $|B|=2$ and $|I\backslash\{\overline{x}\}|=1$. Moreover, the only vertex in $I\backslash\{\overline{x}\}$ has degree two in $G$ because $G$ is simple. Since $G$ is irreducible, each vertex in $B$ has degree three or more in $G$.
Thus, we have $d_{G_1}(x)=d_{G_2}(\overline{x})=|C|\geq4$, and each edge of $C$ is a multiple edge of $G_2$. Then each edge of $C$ is a removable edge of $G_2$.
By Lemma \ref{removable-tight-cut}, we have $RE(G_1)\subseteq RE(G)$.
Since $d_{G_1}(x)\geq4$ and $G$ is irreducible, $G_1$ is irreducible and is different from $K_4$ and $\overline{C_6}$.
By the induction hypothesis, $G_1$ has at least $\Delta_1-2$ removable edges, that is, $|RE(G_1)|\geq \Delta_1-2$.
If $\Delta_1\geq\Delta$, since $RE(G_1)\subseteq RE(G)$, we have $|RE(G)|\geq|RE(G_1)|\geq \Delta_1-2\geq\Delta-2$.
Then $G$ has at least $\Delta-2$ removable edges.
To complete the proof, we now show that $\Delta_1\geq\Delta$. Clearly, it is true when $u\in \overline{X}$.
We may assume that $u\in X$. Then $u\in B$ and $\Delta_1\geq d_{G_1}(x)=d_{G_2}(\overline{x})\geq d_G(u)-1+2=\Delta+1$.
Theorem \ref{main:result} holds.   $\hfill\square$  \\

\textbf{Remark.} The condition of Theorem \ref{main:result} that the graph is irreducible is necessary. For instance, the graph in Figure \ref{fig2}($a$) is a near-brick with maximum degree four but not irreducible, and has exactly one removable edge $e$.
Furthermore, the lower bound of Theorem \ref{main:result} is sharp. The graph
shown in Figure \ref{fig2}($b$) is an irreducible near-brick with maximum degree four and has exactly two removable edges $e$ and $f$; the graph $R_8$ shown in Figure \ref{fig2}($c$) is a cubic brick with exactly one removable edge $h$.

\begin{figure}[h]
 \centering
 \includegraphics[width=0.8\textwidth]{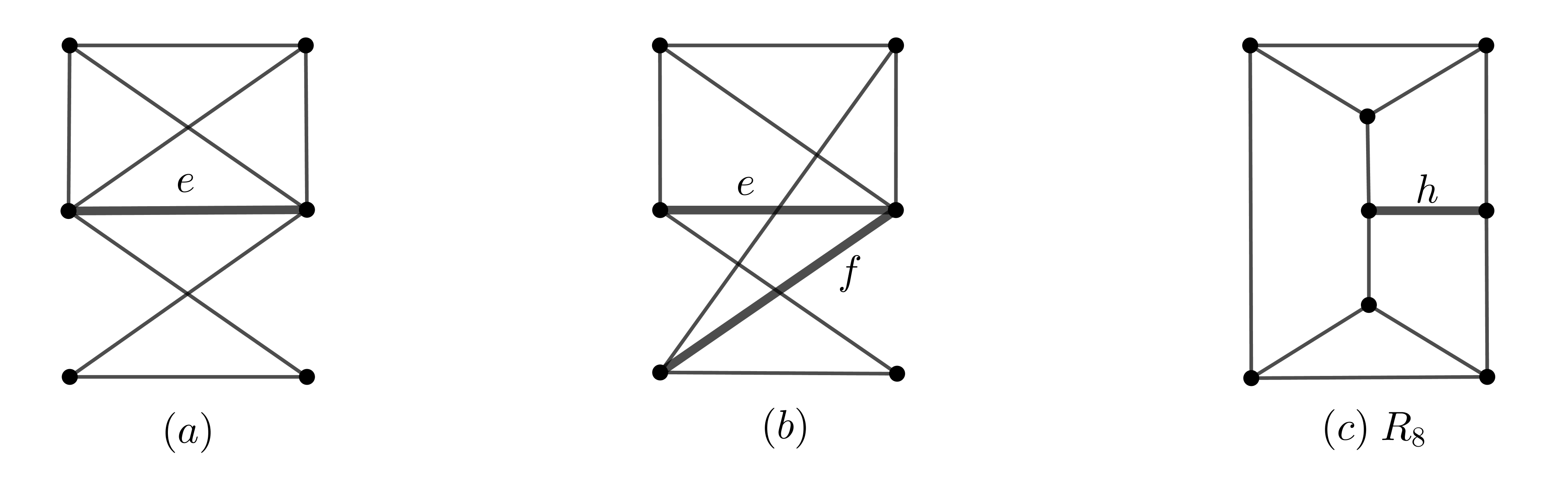}\\
 \caption{The three near-bricks.}\label{fig2}
\end{figure}

\acknowledgements
\label{sec:ack}
This work is supported by the National Natural Science Foundation of China under grant number 12171440.

\nocite{*}
\bibliographystyle{abbrvnat}
\bibliography{sample-dmtcs}
\label{sec:biblio}

\end{document}